\documentclass{article}

\usepackage{amsthm}
\usepackage{amssymb}
\usepackage{amsmath}

\usepackage{a4wide}

\usepackage[round]{natbib}

\newtheorem{theorem}{Theorem}[section]
\newtheorem{proposition}[theorem]{Proposition}
\newtheorem{lemma}[theorem]{Lemma}
\theoremstyle{remark}
\newtheorem{remark}[theorem]{Remark}
\theoremstyle{definition}
\newtheorem{definition}[theorem]{Definition}

\begin{document}

\bibliographystyle{abbrvnat}


\title{Existence and uniqueness results for a fractional\\ 
Riemann--Liouville nonlocal thermistor problem\\
on arbitrary time scales\thanks{This is a preprint of a paper 
whose final and definite form is with 
'Journal of King Saud University -- Science', ISSN 1018-3647. 
Submitted 16-Dec-2016; revised 18-Feb-2017; accepted for publication 14-March-2017.}}

\author{Moulay Rchid Sidi Ammi$^1$\\
\texttt{rachidsidiammi@yahoo.fr}\\[0.3cm]
\and
Delfim F. M. Torres$^2$\\
\texttt{delfim@ua.pt}}

\date{$^1$AMNEA Group, Department of Mathematics,\\
Faculty of Sciences and Techniques, Moulay Ismail University,\\
B.P. 509, Errachidia, Morocco\\[0.3cm]
$^2$Center for Research \& Development in Mathematics and Applications (CIDMA),\\
Department of Mathematics, University of Aveiro, 3810-193 Aveiro, Portugal}

\maketitle


\begin{abstract}
Using a fixed point theorem in a proper Banach space, 
we prove existence and uniqueness results of positive solutions 
for a fractional Riemann--Liouville nonlocal thermistor 
problem on arbitrary nonempty closed subsets of the real numbers.
\end{abstract}


\smallskip


\smallskip


\smallskip


\medskip


\section{Introduction}

The calculus on time scales is a recent area of research introduced 
by \citet{MR1062633}, unifying and extending the theories 
of difference and differential equations into a single theory.
A time scale is a model of time, and the theory has found important applications in several 
contexts that require simultaneous modeling of discrete and continuous data. It is 
under strong current research in areas as diverse as the calculus of variations, 
optimal control, economics, biology, quantum calculus, communication networks 
and robotic control. The interested reader is refereed 
to \citet{agarwal1,agarwal2,bohner1,bohner2,MR2671876,MR3498485} 
and references therein.

On the other hand, many phenomena in engineering, physics and other
sciences, can be successfully modeled by using mathematical
tools inspired by the fractional calculus, that is,
the theory of derivatives and integrals of noninteger order. 
See, for example, \citet{gaul,hilfer,srivastava2,sabatier,samko,srivastava1}. 
This allows one to describe physical phenomena more accurately. 
In this line of thought, fractional differential equations
have emerged in recent years as an interdisciplinary area of research
\citep{MR2962045}. The nonlocal nature of fractional derivatives
can be utilized to simulate accurately diversified natural 
phenomena containing long memory \citep{MR3316531,MR2736622}.

A thermistor is a thermally sensitive resistor whose electrical conductivity changes
drastically by orders of magnitude, as the temperature reaches a certain threshold.
Thermistors are used as temperature control elements in a wide variety of military
and industrial equipment, ranging from space vehicles to air conditioning controllers.
They are also used in the medical field, for localized and general body temperature
measurement; in meteorology, for weather forecasting; as well as in chemical industries,
as process temperature sensors \citep{kwok,maclen}. 

Throughout the reminder of the paper, we denote by $\mathbb{T}$ a time scale, 
which is a nonempty closed subset of $\mathbb{R}$ with its inherited topology. 
For convenience, we make the blanket assumption that $t_{0}$ and $T$ 
are points in $\mathbb{T}$. Our main concern is to prove existence and uniqueness
of solution to a fractional order nonlocal thermistor problem of the form
\begin{equation}
\label{eq1}
\begin{gathered}
 _{t_{0}}^{\mathbb{T}}D^{2 \alpha}_{t} u(t) 
 = \frac{\lambda f(u)}{\left(\int_{t_{0}}^{T} f(u)\, \triangle x\right)^{2}} 
 \, , \quad  t \in (t_{0}, T)  \, , \\
_{t_{0}}^{\mathbb{T}}I_{t}^{\beta}u(t_{0})=0,  
\quad    \forall \, \beta \in  (0, 1),
\end{gathered}
\end{equation}
under suitable conditions on $f$ as described below.
We assume that $\alpha \in (0,1)$ is a parameter describing 
the order of the fractional derivative;
$ _{t_{0}}^{\mathbb{T}}D^{2 \alpha}_{t}$ is
the left Riemann--Liouville fractional derivative operator of order 
$2 \alpha$ on $\mathbb{T}$; $ _{t_{0}}^{\mathbb{T}}I^{\beta}_{t}$ is
the left Riemann--Liouville fractional integral operator 
of order $\beta$ defined on $\mathbb{T}$ by \citet{MyID:328}. 
By $u$, we denote the temperature inside the conductor; 
$f(u)$ is the electrical conductivity of the material.

In the literature, many existence results for dynamic equations on time scales
are available \citep{dogan1,dogan2}. In recent years, there has been also 
significant interest in the use of fractional differential equations
in mathematical modeling \citep{MR3370824,MR3529931,MR3535730}. 
However, much of the work published to date has been 
concerned separately, either by the time-scale community or 
by the fractional one. Results on fractional dynamic equations 
on time scales are scarce \citep{ahmad}.

In contrast with our previous works, which make use of fixed point theorems 
like the Krasnoselskii fixed point theorem, the fixed point index theory, 
and the Leggett--Williams fixed point theorem,  
to obtain several results of existence of positive solutions to linear and
nonlinear dynamic equations on time scales, and recently also to fractional differential
equations \citep{sidiammi3,sidiammi1,sidiammi2,GL}; 
here we prove new existence and uniqueness results for the fractional order
nonlocal thermistor problem on time scales \eqref{eq1},
putting together time scale and fractional domains.
This seems to be quite appropriate from the point 
of view of practical applications \citep{MR3323914,MyID:358,MR3498485}.

The rest of the article is arranged as follows. In Section~\ref{section2}, 
we state preliminary definitions, notations, propositions and properties 
of the fractional operators on time scales needed in the sequel. 
Our main aim is to prove existence of solutions for \eqref{eq1} 
using a fixed point theorem and, consequently, uniqueness. 
This is done in Section~\ref{section3}: 
see Theorems~\ref{thm1} and \ref{corollary}.


\section{Preliminaries}
\label{section2}

In this section, we recall fundamental definitions, hypotheses and preliminary
facts that are used through the paper. For more details, see the seminal paper
\citet{MyID:328}. From physical considerations, we assume that the electrical 
conductivity is bounded \citep{ac}. Precisely, we consider the following assumption:
\begin{itemize}
\item[(H1)] function $f: \mathbb{R}^{+}\rightarrow \mathbb{R}^{+}$ 
of problem \eqref{eq1} is Lipschitz continuous with Lipschitz constant 
$L_{f}$ such that $c_{1} \leq f(u) \leq c_{2}$,
with $c_{1}$ and $c_{2}$ two positive constants.
\end{itemize}
We deal with the notions of left Riemann--Liouville fractional integral 
and derivative on time scales, as proposed in \citet{MyID:328}, the so called 
BHT fractional calculus on time scales \citep{MyID:358}. The corresponding 
right operators are obtained by changing the limits of integrals 
from $a$ to $t$ into $t$ to $b$. For local approaches to fractional calculus 
on arbitrary time scales we refer the reader to \citet{ben,MyID:320}. 
Here we are interested in nonlocal operators, which are the ones who make
sense with respect to the thermistor problem \citep{MR2397904,MR2980253}. 
Although we restrict ourselves to the delta approach on time scales, 
similar results are trivially obtained for the nabla fractional 
case \citep{MyID:224}.

\begin{definition}[Riemann--Liouville fractional integral on time scales \citep{MyID:328}]
\label{def1}
Let $\mathbb{T}$ be a time scale and $[a, b]$ an interval of $\mathbb{T}$.  
Then the left fractional integral on time scales of order $ 0 < \alpha <1$
of a function $g: \mathbb{T} \rightarrow \mathbb{R}$ is defined by
$$
_{a}^{\mathbb{T}}I^{\alpha}_{t} g(t) 
= \int_a^t \frac{(t - s)^{ \alpha-1}}{\Gamma( \alpha)} g(s)\, \triangle s,
$$
where $\Gamma $ is the Euler gamma function.
\end{definition}

The left Riemann--Liouville fractional derivative operator 
of order $\alpha$ on time scales is then defined using
Definition~\ref{def1} of fractional integral.

\begin{definition}[Riemann--Liouville fractional derivative on time scales \citep{MyID:328}]
\label{def2}
Let $\mathbb{T}$ be a time scale, $[a, b]$ an interval of $\mathbb{T}$, 
and $0< \alpha < 1$. Then the left Riemann--Liouville fractional derivative 
on time scales of order $\alpha$ of a function $g: \mathbb{T} \rightarrow \mathbb{R}$ 
is defined by
$$
_{a}^{\mathbb{T}}D^{\alpha}_{t} g(t) 
= \left (\int_a^t \frac{(t - s)^{ -\alpha}}{\Gamma(1- \alpha)} 
g(s)\, \triangle s \right)^{\triangle}.
$$
\end{definition}

\begin{remark}
If $\mathbb{T}= \mathbb{R}$, then we obtain from Definitions~\ref{def1} and \ref{def2}, 
respectively, the usual left Rieman--Liouville fractional integral and derivative.
\end{remark}

\begin{proposition}[See \citet{MyID:328}]
Let $\mathbb{T}$ be a time scale, 
$g: \mathbb{T} \rightarrow \mathbb{R}$ and $0< \alpha < 1$. Then
\[
_{a}^{\mathbb{T}}D^{\alpha}_{t} g
= \triangle \circ _{a}^{\mathbb{T}}I^{1- \alpha}_{t} g.
\]
\end{proposition}

\begin{proposition}[See \citet{MyID:328}]
If $\alpha >0$ and $g \in C([a, b])$, then
\[
_{a}^{\mathbb{T}}D^{\alpha}_{t} 
\circ  _{a}^{\mathbb{T}}I^{\alpha}_{t}g= g.
\]
\end{proposition}

\begin{proposition}[See \citet{MyID:328}]
\label{prop:I:D}
Let $g \in C([a, b])$ and $0 <\alpha < 1$. If 
$_{a}^{\mathbb{T}}I^{1-\alpha}_{t} u(a)=0$, then
\[
_{a}^{\mathbb{T}}I^{\alpha}_{t}
\circ _{a}^{\mathbb{T}}D^{\alpha}_{t}g= g.
\]
\end{proposition}

\begin{theorem}[See \citet{MyID:328}]
Let $g \in C([a, b])$, $\alpha > 0$, and $_{a}^{\mathbb{T}}I^{\alpha}_{t}([a, b])$ 
be the space of functions that can be represented by the Riemann--Liouville 
$\triangle$-integral of order $\alpha$ of some $C([a, b])$-function. Then,
\[
g \in\, _{a}^{\mathbb{T}}I^{\alpha}_{t}([a, b])
\]
if and only if
\[
_{a}^{\mathbb{T}}I^{1-\alpha}_{t} g \in C^{1}([a, b])
\]
and
\[
_{a}^{\mathbb{T}}I^{1-\alpha}_{t} g(a) =0.
\]
\end{theorem}

The following result of the calculus on time scales is also useful.

\begin{proposition}[See \citet{ahmad}] 
\label{proposition25}
Let $\mathbb{T}$ be a time scale and $g$ an increasing continuous 
function on the time-scale interval $[a, b]$. If $G$ is the extension 
of $g$ to the real interval $[a, b]$ defined by
$$
G(s):=
\begin{cases}
g(s) & \mbox{ if } s \in \mathbb{T},\\
g(t) &  \mbox{ if } s \in (t, \sigma(t)) \notin  \mathbb{T},
\end{cases}
$$
then
\[
\int_{a}^{b} g(t) \triangle t \leq  \int_{a}^{b} G(t)  dt,
\]
where $\sigma: \mathbb{T} \rightarrow \mathbb{T}$ is the forward 
jump operator of $\mathbb{T}$ defined by $\sigma(t):= \inf \{ s \in \mathbb{T}: s > t\}$.
\end{proposition}

Along the paper, by $C([0,T])$ we denote the space  of all continuous 
functions on $[0,T]$ endowed with the norm 
$\| x \|_{\infty} =  \sup_{t\in [0,T]} \{  |x(t)| \}$.
Then, $X = \left( C([0,T]), \| \cdot \| \right)$ is a Banach space.


\section{Main Results}
\label{section3}

We begin by giving an integral representation to our problem \eqref{eq1}.
Due to physical considerations, the only relevant case
is the one with $0<\alpha<\frac{1}{2}$. Note that this is coherent 
with our fractional operators with $2 \alpha -1 < 0$.

\begin{lemma}
Let $ 0< \alpha < \frac{1}{2}$. Problem
\eqref{eq1} is equivalent to
\begin{equation}
\label{eq2}
u(t) = \frac{\lambda}{\Gamma (2 \alpha )} \int_{t_{0}}^{t} (t-s)^{2 \alpha -1} 
\frac{f(u(s))}{( \int_{t_{0}}^{T} f(u)\, \triangle x)^{2}} \triangle s \, .
\end{equation}
\end{lemma}

\begin{proof}
We have 
\begin{equation*}
\begin{split}
_{t_{0}}^{\mathbb{T}}D^{2 \alpha}_{t} u(t) 
&=  \frac{\lambda}{\Gamma ( 2 \alpha )} \left( 
\int_{t_{0}}^{t} (t-s)^{2 \alpha -1} \frac{f(u(s))}{\left(
\int_{t_{0}}^{T} f(u)\, \triangle x\right)^{2}}\triangle s\right )^{\triangle}\\
&= \left( _{t_{0}}^{\mathbb{T}}I^{1- 2 \alpha}_{t} u(t) \right )^{\triangle} 
= \left(\triangle \, \circ \, _{t_{0}}^{\mathbb{T}}I^{1- 2 \alpha}_{t} \right) u(t).
\end{split}
\end{equation*}
The result follows from Proposition~\ref{prop:I:D}: 
$_{t_{0}}^{\mathbb{T}}I^{2 \alpha}_{t}  \circ 
\left ( _{t_{0}}^{\mathbb{T}}D^{ 2 \alpha}_{t} (u) \right) = u$.
\end{proof}

For the sake of simplicity, we take $t_{0}=0$. 
It is easy to see that \eqref{eq1} has a solution
$u=u(t)$ if and only if $u$ is a fixed point of the operator 
$K: X \rightarrow X $ defined by
\begin{equation}
\label{eq3}
Ku(t)=  \frac{\lambda}{\Gamma ( 2 \alpha )} \int_{0}^{t} (t-s)^{2 \alpha -1} 
\frac{f(u(s))}{\left(\int_{0}^{T} f(u)\, \triangle x\right)^{2}} \triangle s \, .
\end{equation}

Follows our first main result.

\begin{theorem}[Existence of solution]
\label{thm1}
Let $ 0< \alpha < \frac{1}{2}$
and $f$ satisfies hypothesis $(H1)$.
Then there exists a  solution $u \in X $ of
\eqref{eq1} for all $\lambda > 0$.
\end{theorem}


\subsection{Proof of Existence}

In this subsection we prove Theorem~\ref{thm1}. For that,
firstly we prove that the operator $K$ defined
by \eqref{eq3} verifies the conditions of Schauder's 
fixed point theorem \citep{cronin}.

\begin{lemma}
\label{m:lema}
The operator $K$ is continuous.
\end{lemma}

\begin{proof}
Let us consider a sequence $u_{n}$ converging to $u$ in $X$. Then,
\begin{equation*}
\begin{split}
|K u_{n}&(t) - Ku(t)| 
\leq  \frac{\lambda}{\Gamma (2 \alpha)} 
\int_{0}^{t} (t-s)^{2 \alpha -1}  
\left| \frac{f(u_{n}(s))}{\left( \int_{0}^{T} f(u_{n})\, \triangle x\right)^{2}}
- \frac{f(u(s))}{\left(\int_{0}^{T} f(u)\, \triangle x\right)^{2}}
\, \right| \triangle s  \\
\end{split}
\end{equation*}
\begin{equation}
\label{eq4}
\begin{split}
& \leq  \frac{\lambda}{\Gamma (2 \alpha)} 
\int_{0}^{t} (t-s)^{2 \alpha -1}   
\left|  \frac{1}{\left(\int_{0}^{T} f(u_{n})\, \triangle x\right)^{2}} 
\left(f(u_{n}(s))-f(u(s))\right)\right.\\
& \qquad \left.+ f(u(s))  \left( \frac{1}{\left( \int_{0}^{T} f(u_{n})\, \triangle x\right)^{2}}
- \frac{1}{\left(\int_{0}^{T} f(u)\,  \triangle x \right)^{2}} \right)  \right| \\
& \leq \frac{\lambda}{\Gamma (2 \alpha )} 
\int_{0}^{t} (t-s)^{2 \alpha -1} \frac{1}{\left(\int_{0}^{T} f(u_{n})\, \triangle x\right)^{2}} 
\left|f(u_{n}(s))-f(u(s))\right| \triangle s \\
& \qquad +  \frac{\lambda}{\Gamma (2 \alpha )} 
\int_{0}^{t} (t-s)^{2 \alpha -1} \left|f(u(s))\right| \left|  
\frac{1}{\left(\int_{0}^{T} f(u_{n})\, \triangle x\right)^{2}}  
-  \frac{1}{\left(\int_{0}^{T} f(u)\, \triangle x\right)^{2}} \right| 
\leq  I_{1} + I_{2}.
\end{split}
\end{equation}
We estimate both right-hand terms separately. 
By hypothesis $(H1)$ and Proposition~\ref{proposition25}, we have
\begin{equation*}
\begin{split}
I_{1} & \leq \frac{\lambda}{(c_{1}T)^{2}\Gamma (2 \alpha )}
\int_{0}^{t} (t-s)^{2 \alpha -1}  \left|f(u_{n}(s))-f(u(s))\right| \triangle s \\
& \leq \frac{\lambda L_{f}}{(c_{1}T)^{2}\Gamma (2 \alpha )}
\int_{0}^{t} (t-s)^{2 \alpha -1}  \left|u_{n}(s))-u(s)\right| \triangle s \\
& \leq \frac{\lambda L_{f}}{(c_{1}T)^{2}\Gamma (2 \alpha )}
\| u_{n}-u \|_{\infty} \int_{0}^{t} (t-s)^{2 \alpha -1} \triangle s \\
& \leq \frac{\lambda L_{f}}{(c_{1}T)^{2}\Gamma (2 \alpha )}
\| u_{n}-u \|_{\infty} \int_{0}^{t} (t-s)^{2 \alpha -1} ds,
\end{split}
\end{equation*}
since $(t-s)^{2 \alpha -1}$ is nondecreasing. Then,
\begin{equation}
\label{eq5}
I_{1} \leq \frac{\lambda T^{2 \alpha}L_{f}}{(c_{1}T)^{2}
\Gamma (2 \alpha +1 )}\| u_{n}-u \|_{\infty}.
\end{equation}
Once again, since $(t-s)^{2 \alpha -1}$ is nondecreasing, we have
\begin{equation*}
\begin{split}
I_{2} & \leq \frac{\lambda}{\Gamma (2 \alpha )}   \int_{0}^{t}   
\frac{(t-s)^{2 \alpha -1} | f(u(s))| \left|  
\left( \int_{0}^{T} f(u_{n})\, \triangle x\right)^{2}
-\left( \int_{0}^{T} f(u)\, \triangle x\right)^{2} \right |}{\left( 
\int_{0}^{T} f(u_{n})\, \triangle x\right)^{2} 
\left( \int_{0}^{T} f(u)\, \triangle x\right)^{2}} \triangle s\\
& \leq \frac{\lambda c_{2}}{(c_{1}T)^{4}\Gamma (2 \alpha )}
\int_{0}^{t} (t-s)^{2 \alpha -1} \left|
\left( \int_{0}^{T} f(u_{n})\, \triangle x\right)^{2}
-\left( \int_{0}^{T} f(u)\, \triangle x\right)^{2} \right | \triangle s\\
& \leq \frac{\lambda c_{2}}{(c_{1}T)^{4}\Gamma (2 \alpha )}
\int_{0}^{t} (t-s)^{2 \alpha -1}  
\left|\left( \int_{0}^{T} (f(u_{n})-f(u))\, \triangle x\right)
\left( \int_{0}^{T} (f(u_{n})+f(u))\, \triangle x\right) \right|\triangle s \\
& \leq \frac{2 \lambda c_{2}^{2}T}{(c_{1}T)^{4}\Gamma (2 \alpha )}
\int_{0}^{t} (t-s)^{2 \alpha -1} \left( \int_{0}^{T} \left|f(u_{n})-f(u)\right|
\, \triangle x \right ) \triangle s \\
& \leq \frac{2 \lambda c_{2}^{2}T L_{f}}{(c_{1}T)^{4}\Gamma (2 \alpha )} 
\int_{0}^{t} (t-s)^{2 \alpha -1} \left( \int_{0}^{T} \left|u_{n}(x)-u(x)\right|\, 
\triangle x \right ) \triangle s \\
& \leq \frac{2 \lambda c_{2}^{2}T^{2} L_{f}}{(c_{1}T)^{4}\Gamma (2 \alpha )} 
\| u_{n}-u \|_{\infty}  \int_{0}^{t} (t-s)^{2 \alpha -1}  \triangle s \\
& \leq \frac{2 \lambda c_{2}^{2}T^{2} L_{f}}{(c_{1}T)^{4}\Gamma (2 \alpha )} 
\| u_{n}-u \|_{\infty}  \int_{0}^{t} (t-s)^{2 \alpha -1} ds \\
& \leq \frac{2 \lambda c_{2}^{2}T^{2(\alpha +1)} 
L_{f}}{(c_{1}T)^{4}\Gamma (2 \alpha+1 )} \| u_{n}-u \|_{\infty}.
\end{split}
\end{equation*}
It follows that
\begin{equation}
\label{eq6}
I_{2} \leq \frac{2 \lambda c_{2}^{2}T^{2(\alpha +1)} 
L_{f}}{(c_{1}T)^{4}\Gamma (2 \alpha+1 )} \| u_{n}-u \|_{\infty}.
\end{equation}
Bringing inequalities \eqref{eq5} and \eqref{eq6} into \eqref{eq4}, we have
\begin{equation*}
|K u_{n}(t) - Ku(t)| \leq  I_{1}+I_{2}\\
\leq  \left ( \frac{\lambda T^{2 \alpha}L_{f}}{(c_{1}T)^{2}
\Gamma (2 \alpha +1 )}  +  \frac{2 \lambda c_{2}^{2}T^{2(\alpha +1)} 
L_{f}}{(c_{1}T)^{4}\Gamma (2 \alpha+1 )} \right)   \| u_{n}-u \|_{\infty}.
\end{equation*}
Then
\begin{equation}
\label{eq7}
\|K u_{n} - Ku\|_{\infty} 
\leq  \left ( \frac{\lambda T^{2 \alpha}L_{f}}{(c_{1}T)^{2}\Gamma (2 \alpha +1 )} 
+ \frac{2 \lambda c_{2}^{2}T^{2(\alpha +1)} L_{f}}{(c_{1}T)^{4}
\Gamma (2 \alpha+1 )} \right)\| u_{n}-u \|_{\infty}.
\end{equation}
Hence, independently of $\lambda$, the right-hand side of the above inequality 
converges to zero as $u_{n}\rightarrow u$. Therefore, $K u_{n} \rightarrow  Ku$. 
This proves the continuity of $K$.
\end{proof}

\begin{lemma}
The operator $K$ sends bounded sets into bounded sets on 
$\mathbb{C}([0, T], \mathbb{R})$.
\end{lemma}

\begin{proof}
Let $I= [0, T]$. We need to prove that for all $r>0$ there exists $l > 0$ such that 
for all $u \in B_{r}= \{ u \in \mathbb{C}(I, \mathbb{R}, \|u\|_{\infty} \leq r \}$ 
we have $\|K(u)\|_{\infty} \leq l$. Let $t \in I$ and $u \in B_{r}$. Then,
\begin{equation*}
\begin{split}
|K(u(t))| 
&  \leq \frac{\lambda}{\Gamma (2 \alpha )} 
\int_{0}^{t} (t-s)^{2 \alpha -1} 
\frac{| f(u(s))|}{\left( \int_{0}^{T} f(u)\, \triangle x\right)^{2}} \triangle s\\
&\leq \frac{\lambda M}{\Gamma (2 \alpha )}   \int_{0}^{t} (t-s)^{2 \alpha -1}  \triangle s\\
&\leq \frac{\lambda M}{\Gamma (2 \alpha )}   \int_{0}^{t} (t-s)^{2 \alpha -1}  ds\\
&\leq \frac{\lambda M T^{2\alpha}}{\Gamma (2 \alpha +1 )},
\end{split}
\end{equation*}
where $M=\frac{\sup_{B_{r}} f}{(c_{1}T)^{2}} $. Hence,
taking the supremum over $t$, it follows that
$$
\|K(u)\|_{\infty} \leq \frac{\lambda M T^{2\alpha}}{\Gamma (2 \alpha +1 )},
$$
that is, $K(u)$ is bounded.
\end{proof}

Now, we shall prove that $K(B_{r})$ is an equicontinuous set in $X$.
This ends the proof of our Theorem~\ref{thm1}.

\begin{lemma}
The operator $K$ sends bounded sets into equicontinuous 
sets of $\mathbb{C}(I, \mathbb{R})$.
\end{lemma}

\begin{proof}
Let $t_{1}, t_{2} \in I$ such that $0 \leq t_{1} < t_{2} \leq T$, 
$B_{r}$ is a bounded set of $\mathbb{C}(I, \mathbb{R})$ and $u \in B_{r}$. Then,
\begin{equation*}
\begin{split}
|K(u(t_{2}))- K(u(t_{1}))|
&\leq \frac{\lambda}{\Gamma (2 \alpha )} 
\left| \int_{0}^{t_{2}} \left(t_{2}-s\right)^{2 \alpha -1} 
\frac{f(u(s))}{\left(\int_{0}^{T} f(u)\, \triangle x\right)^{2}} \triangle s\right. \\
& \qquad \left. - \int_{0}^{t_{1}} \left(t_{1}-s\right)^{2 \alpha -1} 
\frac{f(u(s))}{\left(\int_{0}^{T} f(u)\, \triangle x\right)^{2}} \triangle s \right|\\
\end{split}
\end{equation*}
\begin{equation*}
\begin{split}
&  \leq \frac{\lambda}{\Gamma (2 \alpha)} \left| 
\int_{0}^{t_{2}} \left(\left(t_{2}-s\right)^{2 \alpha -1} 
- \left(t_{1}-s\right)^{2 \alpha -1} + \left(t_{1}-s\right)^{2 \alpha -1} \right)  
\frac{f(u(s))}{\left(\int_{0}^{T} f(u)\, \triangle x\right)^{2}} \triangle s \right.\\
&\qquad \left. - \int_{0}^{t_{1}} \left(t_{1}-s\right)^{2 \alpha -1} 
\frac{f(u(s))}{\left(\int_{0}^{T} f(u)\, \triangle x\right)^{2}} \triangle s \right|\\
& \leq \frac{\lambda}{\Gamma (2 \alpha)} \left| 
\int_{0}^{t_{2}} \left(\left(t_{2}-s\right)^{2 \alpha -1} 
- \left(t_{1}-s\right)^{2 \alpha -1}\right) 
\frac{f(u(s))}{\left(\int_{0}^{T} f(u)\, \triangle x\right)^{2}} \triangle s\right. \\
&\qquad \left. + \int_{t_{1}}^{t_{2}} \left(t_{1}-s\right)^{2 \alpha -1}
\frac{f(u(s))}{\left(\int_{0}^{T} f(u)\, \triangle x\right)^{2}} \triangle s \right|\\
&  \leq \frac{\lambda c_{2}}{\left(c_{1}T\right)^{2}\Gamma (2 \alpha )} 
\left| \int_{0}^{t_{2}} \left(\left(t_{2}-s\right)^{2 \alpha -1} 
- \left(t_{1}-s\right)^{2 \alpha -1} \right) ds 
+ \int_{t_{1}}^{t_{2}} \left(t_{1}-s\right)^{2 \alpha -1} ds \right|\\
& \leq \frac{\lambda c_{2}}{\left(c_{1}T\right)^{2}\Gamma \left(2 \alpha +1\right)} 
\left| t_{1}^{2 \alpha} - t_{2}^{2 \alpha} 
+ \left(t_{1}-t_{2}\right)^{2 \alpha} - \left(t_{1}-t_{2}\right)^{2 \alpha}\right| \\
&\leq \frac{\lambda c_{2}}{\left(c_{1}T\right)^{2}\Gamma (2 \alpha +1)} 
\left| t_{2}^{2 \alpha} - t_{1}^{2 \alpha}\right|.
\end{split}
\end{equation*}
Because the right-hand side of the above inequality does not depend
on $u$ and tends to zero when $t_{2}\rightarrow t_{1}$, we
conclude that $K(\overline{B_{r}})$ is relatively compact.
Hence, $B$ is compact by the Arzela--Ascoli theorem. Consequently, since
$K$ is continuous, it follows by Schauder's fixed point theorem \citep{cronin}
that problem \eqref{eq1} has a solution on $I$. 
This ends the proof of Theorem~\ref{thm1}.
\end{proof}


\subsection{Uniqueness}

We now derive uniqueness of solution to problem \eqref{eq1}.

\begin{theorem}[Uniqueness of solution]
\label{corollary}
Let $ 0< \alpha < \frac{1}{2}$ and $f$ satisfies  hypothesis $(H1)$. If
$$
0 < \lambda <  \left( 
\frac{ T^{2 \alpha}L_{f}}{(c_{1}T)^{2}\Gamma (2 \alpha +1 )}  
+ \frac{2 c_{2}^{2}T^{2(\alpha +1)} 
L_{f}}{(c_{1}T)^{4}\Gamma (2 \alpha+1 )} \right)^{-1},
$$
then the solution predicted by Theorem~\ref{thm1} is unique.
\end{theorem}

\begin{proof}
Let $u$ and $v$ be two solutions of \eqref{eq1}. 
Then, from \eqref{eq7}, one has
\begin{equation*}
\|K v - Ku\|_{\infty} 
\leq  \left ( \frac{\lambda T^{2 \alpha}L_{f}}{(c_{1}T)^{2}\Gamma (2 \alpha +1 )}  
+ \frac{2 \lambda c_{2}^{2}T^{2(\alpha +1)} 
L_{f}}{(c_{1}T)^{4}\Gamma (2 \alpha+1 )} \right) \| v-u \|_{\infty}.
\end{equation*}
Choosing $\lambda$ such that 
$0 < \lambda <  \left ( \frac{ T^{2 \alpha}L_{f}}{(c_{1}T)^{2}\Gamma (2 \alpha +1 )}  
+ \frac{2 c_{2}^{2}T^{2(\alpha +1)} L_{f}}{(c_{1}T)^{4}\Gamma (2 \alpha+1 )} \right)^{-1}$,
the map $ K : X \to X $ is a contraction. It follows by the Banach principle 
that it has a fixed point $u=Fu$. Hnce, there exists a unique 
$u \in X$ that is solution of \eqref{eq2}.
\end{proof}


\section*{Acknowledgements}

The authors were supported by the \emph{Center for Research
and Development in Mathematics and Applications} (CIDMA)
of the University of Aveiro, through Funda\c{c}\~ao 
para a Ci\^encia e a Tecnologia (FCT), 
within project UID/MAT/04106/2013. They are grateful
to two anonymous referees for several comments
and suggestions.


\small




\begin{thebibliography}{38}
	
\bibitem[Abbas et~al.(2012)Abbas, Benchohra, and N'Gu\'er\'ekata]{MR2962045}
S.~Abbas, M.~Benchohra, and G.~M. N'Gu\'er\'ekata.
\newblock \emph{Topics in fractional differential equations}, volume~27 of
\emph{Developments in Mathematics}.
\newblock Springer, New York, 2012.
	
\bibitem[Agarwal et~al.(2002)Agarwal, Bohner, O'Regan, and Peterson]{agarwal2}
R.~Agarwal, M.~Bohner, D.~O'Regan, and A.~Peterson.
\newblock Dynamic equations on time scales: a survey.
\newblock \emph{J. Comput. Appl. Math.}, 141\penalty0 (1-2):\penalty0 1--26, 2002.
	
\bibitem[Agarwal and Bohner(1999)]{agarwal1}
R.~P. Agarwal and M.~Bohner.
\newblock Basic calculus on time scales and some of its applications.
\newblock \emph{Results Math.}, 35\penalty0 (1-2):\penalty0 3--22, 1999.
	
\bibitem[Aghababa(2015)]{MR3370824}
M.~P. Aghababa.
\newblock Fractional modeling and control of a complex nonlinear energy
supply-demand system.
\newblock \emph{Complexity}, 20\penalty0 (6):\penalty0 74--86, 2015.
	
\bibitem[Ahmadkhanlu and Jahanshahi(2012)]{ahmad}
A.~Ahmadkhanlu and M.~Jahanshahi.
\newblock On the existence and uniqueness of solution of initial value problem
for fractional order differential equations on time scales.
\newblock \emph{Bull. Iranian Math. Soc.}, 38\penalty0 (1):\penalty0 241--252, 2012.
	
\bibitem[Antontsev and Chipot(1994)]{ac}
S.~N. Antontsev and M.~Chipot.
\newblock The thermistor problem: existence, smoothness uniqueness, blowup.
\newblock \emph{SIAM J. Math. Anal.}, 25\penalty0 (4):\penalty0 1128--1156, 1994.
	
\bibitem[Aulbach and Hilger(1990)]{MR1062633}
B.~Aulbach and S.~Hilger.
\newblock A unified approach to continuous and discrete dynamics.
\newblock In \emph{Qualitative theory of differential equations 
({S}zeged, 1988)}, volume~53 of \emph{Colloq. Math. Soc. J\'anos Bolyai}, pages 37--56.
North-Holland, Amsterdam, 1990.
	
\bibitem[Benkhettou et~al.(2015)Benkhettou, Brito~da Cruz, and Torres]{ben}
N.~Benkhettou, A.~M.~C. Brito~da Cruz, and D.~F.~M. Torres.
\newblock A fractional calculus on arbitrary time scales: fractional
differentiation and fractional integration.
\newblock \emph{Signal Process.}, 107:\penalty0 230--237, 2015.
\newblock {\tt arXiv:1405.2813}
	
\bibitem[Benkhettou et~al.(2016{\natexlab{a}})Benkhettou, Brito~da Cruz, and Torres]{MyID:320}
N.~Benkhettou, A.~M.~C. Brito~da Cruz, and D.~F.~M. Torres.
\newblock Nonsymmetric and symmetric fractional calculi on arbitrary nonempty closed sets.
\newblock \emph{Math. Methods Appl. Sci.}, 39\penalty0 (2):\penalty0 261--279, 2016{\natexlab{a}}.
\newblock {\tt arXiv:1502.07277}
	
\bibitem[Benkhettou et~al.(2016{\natexlab{b}})Benkhettou, Hammoudi, and Torres]{MyID:328}
N.~Benkhettou, A.~Hammoudi, and D.~F.~M. Torres.
\newblock Existence and uniqueness of solution for a fractional
{R}iemann--{L}iouville initial value problem on time scales.
\newblock \emph{J. King Saud Univ. Sci.}, 28\penalty0 (1):\penalty0 87--92, 2016{\natexlab{b}}.
\newblock {\tt arXiv:1508.00754}
	
\bibitem[Bohner and Peterson(2001{\natexlab{a}})]{bohner1}
M.~Bohner and A.~Peterson.
\newblock \emph{Dynamic equations on time scales}.
\newblock Birkh\"auser Boston, Inc., Boston, MA, 2001{\natexlab{a}}.
	
\bibitem[Bohner and Peterson(2001{\natexlab{b}})]{bohner2}
M.~Bohner and A.~Peterson.
\newblock \emph{Dynamic equations on time scales}.
\newblock Birkh\"auser Boston, Inc., Boston, MA, 2001{\natexlab{b}}.
	
\bibitem[Cronin(1994)]{cronin}
J.~Cronin.
\newblock \emph{Differential equations}, volume 180 
of \emph{Monographs and Textbooks in Pure and Applied Mathematics}.
\newblock Marcel Dekker, Inc., New York, second edition, 1994.
	
\bibitem[Debbouche and Torres(2015)]{MR3316531}
A.~Debbouche and D.~F.~M. Torres.
\newblock Sobolev type fractional dynamic equations and optimal multi-integral
controls with fractional nonlocal conditions.
\newblock \emph{Fract. Calc. Appl. Anal.}, 18\penalty0 (1):\penalty0 95--121, 2015.
\newblock {\tt arXiv:1409.6028}
	
\bibitem[Dogan(2013{\natexlab{a}})]{dogan1}
A.~Dogan.
\newblock Existence of three positive solutions for an {$m$}-point
boundary-value problem on time scales.
\newblock \emph{Electron. J. Differential Equations}, No. 149, 10 pages, 2013{\natexlab{a}}.
	
\bibitem[Dogan(2013{\natexlab{b}})]{dogan2}
A.~Dogan.
\newblock Existence of multiple positive solutions for {$p$}-{L}aplacian
multipoint boundary value problems on time scales.
\newblock \emph{Adv. Difference Equ.}, pages 2013:238, 23, 2013{\natexlab{b}}.
	
\bibitem[Gaul et~al.(1991)Gaul, Klein, and Kempfle]{gaul}
L.~Gaul, P.~Klein, and S.~Kempfle.
\newblock Damping description involving fractional operators.
\newblock \emph{Mech. Syst. Signal Process.}, 5:\penalty0 81--88, 1991.
	
\bibitem[Girejko and Torres(2012)]{MyID:224}
E.~Girejko and D.~F.~M. Torres.
\newblock The existence of solutions for dynamic inclusions on time scales via duality.
\newblock \emph{Appl. Math. Lett.}, 25\penalty0 (11):\penalty0 1632--1637, 2012.
\newblock {\tt arXiv:1201.4495}
	
\bibitem[Hilfer(2000)]{hilfer}
R.~Hilfer, editor.
\newblock \emph{Applications of fractional calculus in physics}.
\newblock World Scientific Publishing Co., Inc., River Edge, NJ, 2000.
	
\bibitem[Kilbas et~al.(2006)Kilbas, Srivastava, and Trujillo]{srivastava2}
A.~A. Kilbas, H.~M. Srivastava, and J.~J. Trujillo.
\newblock \emph{Theory and applications of fractional differential equations},
volume 204 of \emph{North-Holland Mathematics Studies}.
\newblock Elsevier Science B.V., Amsterdam, 2006.
	
\bibitem[Kwok(1995)]{kwok}
K.~Kwok.
\newblock \emph{Complete guide to semiconductor devices}.
\newblock McGraw-Hill, New York, 1995.
	
\bibitem[Ma et~al.(2016)Ma, Zhou, Li, and Chen]{MR3529931}
Y.~Ma, X.~Zhou, B.~Li, and H.~Chen.
\newblock Fractional modeling and {SOC} estimation of lithium-ion battery.
\newblock \emph{IEEE/CAA J. Autom. Sin.}, 3\penalty0 (3):\penalty0 281--287, 2016.
	
\bibitem[Machado et~al.(2011)Machado, Kiryakova, and Mainardi]{MR2736622}
J.~T. Machado, V.~Kiryakova, and F.~Mainardi.
\newblock Recent history of fractional calculus.
\newblock \emph{Commun. Nonlinear Sci. Numer. Simul.}, 16\penalty0 (3):\penalty0 1140--1153, 2011.
	
\bibitem[Machado et~al.(2015)Machado, Mainardi, and Kiryakova]{MR3323914}
J.~T. Machado, F.~Mainardi, and V.~Kiryakova.
\newblock Fractional calculus: quo vadimus? ({W}here are we going?).
\newblock \emph{Fract. Calc. Appl. Anal.}, 18\penalty0 (2):\penalty0 495--526, 2015.
	
\bibitem[Maclen(1979)]{maclen}
E.~D. Maclen.
\newblock \emph{Thermistors}.
\newblock Electrochemical publication, Glasgow, 1979.
	
\bibitem[Martins and Torres(2009)]{MR2671876}
N.~Martins and D.~F.~M. Torres.
\newblock Calculus of variations on time scales with nabla derivatives.
\newblock \emph{Nonlinear Anal.}, 71\penalty0 (12):\penalty0 e763--e773, 2009.
\newblock {\tt arXiv:0807.2596}
	
\bibitem[Nwaeze and Torres(in press)]{MyID:358}
E.~R. Nwaeze and D.~F.~M. Torres.
\newblock Chain rules and inequalities for the BHT fractional calculus on arbitrary time scales.
\newblock \emph{Arab J. Math. (Springer)}, 6\penalty0 (1):\penalty0 13--20, 2017.
\newblock {\tt arXiv:1611.09049}
	
\bibitem[Ortigueira et~al.(2016)Ortigueira, Torres, and Trujillo]{MR3498485}
M.~D. Ortigueira, D.~F.~M. Torres, and J.~J. Trujillo.
\newblock Exponentials and {L}aplace transforms on nonuniform time scales.
\newblock \emph{Commun. Nonlinear Sci. Numer. Simul.}, 39:\penalty0 252--270, 2016.
\newblock {\tt arXiv:1603.04410}
	
\bibitem[Sabatier et~al.(2007)Sabatier, Agrawal, and Machado]{sabatier}
J.~Sabatier, O.~P. Agrawal, and J.~A.~T. Machado, editors.
\newblock \emph{Advances in fractional calculus}.
\newblock Springer, Dordrecht, 2007.
	
\bibitem[Samko et~al.(1993)Samko, Kilbas, and Marichev]{samko}
S.~G. Samko, A.~A. Kilbas, and O.~I. Marichev.
\newblock \emph{Fractional integrals and derivatives}.
\newblock Gordon and Breach Science Publishers, Yverdon, 1993.
	
\bibitem[Sidi~Ammi and Torres(2008)]{MR2397904}
M.~R. Sidi~Ammi and D.~F.~M. Torres.
\newblock Numerical analysis of a nonlocal parabolic problem resulting from thermistor problem.
\newblock \emph{Math. Comput. Simulation}, 77\penalty0 (2-3):\penalty0 291--300, 2008.
\newblock {\tt arXiv:0709.0129}
	
\bibitem[Sidi~Ammi and Torres(2012{\natexlab{a}})]{MR2980253}
M.~R. Sidi~Ammi and D.~F.~M. Torres.
\newblock Optimal control of nonlocal thermistor equations.
\newblock \emph{Internat. J. Control}, 85\penalty0 (11):\penalty0 1789--1801, 2012{\natexlab{a}}.
\newblock {\tt arXiv:1206.2873}
	
\bibitem[Sidi~Ammi and Torres(2012{\natexlab{b}})]{sidiammi1}
M.~R. Sidi~Ammi and D.~F.~M. Torres.
\newblock Existence and uniqueness of a positive solution to generalized
nonlocal thermistor problems with fractional-order derivatives.
\newblock \emph{Differ. Equ. Appl.}, 4\penalty0 (2):\penalty0 267--276, 2012{\natexlab{b}}.
\newblock {\tt arXiv:1110.4922}
	
\bibitem[Sidi~Ammi and Torres(2013)]{sidiammi2}
M.~R. Sidi~Ammi and D.~F.~M. Torres.
\newblock Existence of three positive solutions to some {$p$}-{L}aplacian boundary value problems.
\newblock \emph{Discrete Dyn. Nat. Soc.}, pages Art. ID 145050, 12, 2013.
\newblock {\tt arXiv:1210.5351}
	
\bibitem[Sidi~Ammi et~al.(2012)Sidi~Ammi, El~Kinani, and Torres]{sidiammi3}
M.~R. Sidi~Ammi, E.~H. El~Kinani, and D.~F.~M. Torres.
\newblock Existence and uniqueness of solutions to functional integro-differential fractional equations.
\newblock \emph{Electron. J. Differential Equations}, pages No. 103, 9, 2012.
\newblock {\tt arXiv:1206.3996}
	
\bibitem[Souahi et~al.(2016)Souahi, Guezane-Lakoud, and Khaldi]{GL}
A.~Souahi, A.~Guezane-Lakoud, and R.~Khaldi.
\newblock On some existence and uniqueness results for a class 
of equations of order {$0<\alpha\leq 1$} on arbitrary time scales.
\newblock \emph{Int. J. Differ. Equ.}, pages Art. ID 7327319, 8, 2016.
	
\bibitem[Srivastava and Saxena(2001)]{srivastava1}
H.~M. Srivastava and R.~K. Saxena.
\newblock Operators of fractional integration and their applications.
\newblock \emph{Appl. Math. Comput.}, 118\penalty0 (1):\penalty0 1--52, 2001.
	
\bibitem[Yu et~al.(2016)Yu, Perdikaris, and Karniadakis]{MR3535730}
Y.~Yu, P.~Perdikaris, and G.~E. Karniadakis.
\newblock Fractional modeling of viscoelasticity in 3{D} cerebral arteries and aneurysms.
\newblock \emph{J. Comput. Phys.}, 323:\penalty0 219--242, 2016.
	
\end{thebibliography}
\end{document}